\documentclass[10pt,a4paper, xcolor]{article}

\usepackage{enumerate,lineno}
\usepackage{url,hyperref}
\usepackage{amsfonts}
\usepackage{mathtools,bbm}
\usepackage[latin1]{inputenc}

\usepackage{algorithm}
\usepackage[noend]{algpseudocode}

\usepackage{amsfonts}
\usepackage{amssymb,amsthm, amsmath,graphicx,bm,ebezier}

\newtheorem{theorem}{Theorem}
\newtheorem{definition}[theorem]{Definition}
\newtheorem{proposition}[theorem]{Proposition}
\newtheorem{corollary}[theorem]{Corollary}
\newtheorem{lemma}[theorem]{Lemma}

\theoremstyle{remark}

\newtheorem{example}[theorem]{Example}
\newtheorem{remark}[theorem]{Remark}

\def\CaJ{\mathcal{J}}

\def\FraC{\mathcal{C}}

\def\k{\mathbbmss{k}}
\def\N{\mathbb{N}}
\def\R{\mathbb{R}}
\def\Z{\mathbb{Z}}
\def\Q{\mathbb{Q}}

\def\Tint{\mathrm{Tint}}
\def\ap{\mathrm{Ap} }
\def\k{\mathbbmss{k}}

\renewcommand{\u}{{\mathfrak u}}
\def\w{{\mathfrak w}}
\newcommand{\G}{{\mathcal G}}

\title{Proportionally modular affine semigroups}

\author{J. I. Garc\'{\i}a-Garc\'{\i}a\footnote{Departamento de Matem\'aticas, Universidad de C\'adiz,
E-11510 Puerto Real (C\'{a}diz, Spain). E-mail: ignacio.garcia@uca.es. Partially supported by MTM2014-55367-P (MINECO, Spain), FEDER funds and Junta de Andaluc\'{\i}a group FQM-366. }\\
M. A. Moreno-Fr\'{\i}as\footnote{Departamento de Matem\'aticas, Universidad de C\'adiz,
E-11510 Puerto Real (C\'{a}diz, Spain). E-mail: mariangeles.moreno@uca.es.
Partially supported by MTM2014-55367-P (MINECO, Spain), FEDER funds and Junta de Andaluc\'{\i}a group FQM-298.}\\
A. Vigneron-Tenorio\footnote{Departamento de Matem\'aticas, Universidad de C\'adiz,
E-11406 Jerez de la Frontera (C\'{a}diz, Spain). E-mail: alberto.vigneron@uca.es. Partially supported by MTM2015-65764-C3-1-P (MINECO/FEDER, UE) and Junta de Andaluc\'{\i}a group FQM-366.}\\
}

\date{}

\begin{document}
	
\maketitle
	
\begin{abstract}

This work introduces a new kind of semigroup of $\N^p$ called proportionally modular affine semigroup.
These semigroups are defined by modular Diophantine inequalities and they are a generalization of proportionally modular numerical semigroups.
We give an algorithm to compute their minimal generating sets.
We also specialize on the case $p=2$. For this case,
we provide a faster algorithm to compute their minimal system of generators, prove they are Cohen-Macaulay and Buchsbaum, and determinate their (minimal) Frobenius vectors.
Besides, Gorenstein proportionally modular affine semigroups are characterized.

\smallskip
{\small \emph{Keywords:} affine semigroup, Buchsbaum ring, Cohen-Macaulay ring, Frobenius vector, Gorenstein ring, modular Diophantine inequalities, numerical monoid, numerical semigroup}.
		
\smallskip
{\small \emph{MSC-class 2010:} 20M14 (Primary), 11Y50, 13H10, 11D07 (Secondary).}
		
\end{abstract}
	
\section*{Introduction}
All monoids and semigroups appearing in this work are commutative. For this reason, we omit this adjective in the sequel.

Given two non-negative integers $a,b$ with $b\neq 0$, we denote by $a\mod b$ the remainder of the Euclidean division of $a$ by $b.$
For a rational number $p/q$ with $\gcd (p,q)=1,$ we say that $p/q= 0 \mod b$ if $p=0 \mod b,$ and for two rational numbers $p/q$ and $p'/q'$ we say that $p/q = p'/q' \mod b$ if $p/q - p'/q' = 0 \mod b.$
A proportionally modular Diophantine inequality in one variable is an expression of the form $a x\mod b\leq cx$ with $a$, $b$ and $c$ positive integers.
The set $S$ of non-negative integer solutions of that modular inequality is a numerical semigroup, that is, it is a subset of the set of non-negative integers $\N$ that is closed under addition, $0\in S$ and $\N\setminus S$ has finitely many elements. So, the submonoids of $\N$ of the form $S=\{x\in \N~|~ a x\mod b\leq cx\}$ are called proportionally modular numerical semigroups.
They were introduced in \cite{MR2020273}, and many papers about them have been written (see, for example
\cite{MR3231518}, \cite{MR2753834}, \cite{MR2389850}, \cite{MR2796773}).

In this work, we
introduce proportionally modular affine semigroups as a generalization of proportionally modular numerical semigroups.
Instead of using three integers $a,b$ and $c$, we use
two nonnull linear functions $f,g:\Q^p\to \Q$ and a natural number $b.$
With these elements, we define the semigroup
\[S=\{x\in\N^p~|~f(x)\mod b\leq g(x)\}\]
which it is called proportionally modular affine semigroup.
In Theorem \ref{thfg}, we provide an algorithm to compute the minimal generating of $S.$
Besides, we prove that the intersection of every
rational straight line with a proportionally modular affine semigroup is
isomorphic to a proportionally modular numerical semigroup.
This makes possible to view every proportionally modular affine semigroup as a beam of proportionally modular numerical semigroups.

In the second part of this work, we focus on
proportionally modular affine semigroups of $\N^2$ and we explore some of their properties.
Using that any nontrivial proportionally modular affine semigroup $S$ of $\N^2$ is simplicial, we study some of the properties of its associated semigroup ring $\k[S]$ from the semigroup $S.$
These properties are the Cohen-Macaulayness, Gorensteiness and Buchsbaumness, which
have been widely studied in ring theory, but if we try to search
these kind of rings, few methods to obtain them are found (see \cite{RosalesBuchs}, \cite{MR1732040}, \cite{RosalesCM}, \cite{MR881220} and references therein).
In particular, we prove that these semigroups are Cohen-Macaulay and Buchsbaum, and we characterize when they are Gorenstein.
That allows us to affirm that an application of modular Diophantine inequalities to Commutative Algebra is the construction of special kinds of rings.
For these semigroups, we also give a geometrical algorithm for a fast computation of the minimal generating set of $S,$ and we determinate the (minimal) Frobenius vectors of these semigroups (some references to Frobenius vectors are found in \cite{Assi} and \cite{1311.1988}).
In this work, all the examples have been done using the software available at \cite{ProporcionallyModularAffineSemigroupN2}.

The content of this work is organized as follows:
in Section \ref{preliminaires}, we provide some basic definitions and results on monoids and semigroup rings.
In Section \ref{p-m-a-s}, we give the definition of proportionally modular semigroup, we represent it as beams of proportionally numerical semigroups, and we show an algorithm to compute its minimal generating set.
In Section \ref{n2}, we provide a faster algorithm to obtain a system of generators of a given proportionally modular affine semigroup of $\N^2$. For this semigroup, we also give its set of (minimal) Frobenius vectors. Finally, in Section \ref{propiedades_n2}, we study methods to check the above mentioned properties of the semigroup ring of a proportionally modular affine semigroup of $\N^2.$

\section{Preliminaries and notations}\label{preliminaires}

A semigroup is a pair $(S,+)$, with $S$ a nonempty set and $+$ a binary operation defined on $S$ verifying the associative law. In addition, if there exists an element $0\in S$ such that $a+0=0+a$ for all $a\in S$, we say that $(S,+)$ is a monoid. Given a subset $A$ of a monoid $S$, the monoid generated by $A$, denoted by $\langle A\rangle$, is the least (with respect to inclusion) submonoid of $S$ containing $A$. When $S=\langle A\rangle$, we say that $S$ is generated by $A$ or that $A$ is a system of generators of $S$. The monoid $S$ is finitely generated if it has a finite generating set. Finitely generated submonoids of $\N^p$ are known as affine semigroups, and they are called numerical semigroups when $p=1.$

Given a system of linear equations or linear inequalities, a solution is called $\N$-solution if it is a non-negative integer solution (see \cite{Nsolutions} for details).
For a subset $A\in\Q^p$, denote by $\text{ConvexHull}(A)$ the convex hull of the set $A$, that is, the smallest convex subset of $\Q^p$ containing $A$.

In this work, the product ordering in $\N^p$ is denoted by $\preceq.$ So, given two elements $x,y\in \N^p$, $x\preceq y$ if $y-x\in \N^p.$ Besides, we denote by $||x||$ the 1-norm of $x$ (i.e. $||x||=\sum _{i=1}^p |x_i|$) and by $[k]$ the set $\{1,\ldots ,k\}$ for every $k\in\N.$
We use
${\rm L}(A)$ to denote the set
$\{ \sum_{i=1}^m \lambda_i a_i~|~ \lambda_i\in \Q_\geq,~a_i\in A,~m\in \N\}$, this set is known as
the rational cone of $A$.

For a better understanding of the last section of this work, we need to recall some definitions. Let $R$ be a Noetherian local ring, a finite $R$-module
$M\neq0$  is a Cohen-Macaulay module if ${\rm depth}(M)=\dim(M)$. If $R$ itself
is a Cohen-Macaulay module, then it is called a Cohen-Macaulay ring (see \cite{MR1251956}).
A Gorenstein ring is a special case of Cohen-Macaulay ring: a Gorenstein local ring is a Noetherian commutative local ring $R$ with finite injective dimension, as an $R$-module (see \cite{MR0153708}).
The last concept, Buchsbaum ring, is defined as follows: a noetherian $R$-module $M$ is called a Buchsbaum module is every system of parameters of $M$ is a weak $M$-sequence, and $R$ is a Buchsbaum ring if it is Buchsbaum module as a module over itself (see \cite{MR602063} and \cite{MR881220} for details).
For every $(S,+)$ finitely generated commutative monoid and a field $\k,$  we denote by $\k[S]$ the semigroup ring of $S$ over $\k$.
Note that $\k[S]$ is equal to $\bigoplus_{m \in S} \mathbbmss{k} \chi^m$ endowed with a multiplication which is $\mathbbmss{k}$-linear and such that $\chi^m \cdot \chi^n = \chi^{m+n}$ for every $m,n \in S$ (see \cite{MR1600261}).
We say that $S$ is a  Cohen-Macaulay/Gorenstein/Buchsbaum semigroup if $\k[S]$ is a Cohen-Macaulay/Gorenstein/Buchsbaum ring.

\section{Proportionally modular affine semigroups}\label{p-m-a-s}

Let $f,g:\Q^p\to \Q$ be two nonnull linear functions and let $b\in\N\setminus\{0\}$.
If $x,y\in\N^p$ verify $f(x)\mod b\leq g(x)$ and $f(y)\mod b\leq g(y)$, by the linearity
of $f$ and $g$, we have
$f(x+y)\mod b = (f(x)+f(y))\mod b\leq (f(x)\mod b) +(f(y)\mod b)\leq g(x)+g(y)=g(x+y)$.
Clearly, $f(0)\mod b=0\leq g(0)$.
In this way, the set of $\N-$solutions of every inequality of the form $f(x)\mod b\leq g(x)$ is a submonoid of $\N^p$.
Every submonoid $S$ of $\N^p$ obtained as above is called a	 proportionally modular monoid.

Let us suppose that $f(x_1,\dots,x_p)=f_1x_1+\dots+f_px_p$ and $g(x_1,\dots,x_p)=g_1x_1+\dots+g_px_p$ with $f_i,g_i\in \Q$ for all $i=1,\dots,p$, and $b\in \Q_\geq$.
If $d\in\N$, then an element $x\in\N^p$ verifies $f(x)\mod b\leq g(x)$ if and only if
$df(x)\mod db\leq dg(x)$.
Hence, by multiplying the inequality $f(x)\mod b\leq g(x)$ by $d$ the least common multiple of the denominators of
$f_1,\dots,f_p,g_1,\dots,g_p,$ and $b$, we obtain an inequality where $db\in\N$ and the coefficients of
$df$ and $dg$ are integers.
So, in the sequel, we assume that
$f(x_1,\dots,x_p)=f_1x_1+\dots+f_px_p$ and $g(x_1,\dots,x_p)=g_1x_1+\dots+g_px_p$ with $f_i,g_i\in\Z$
and $b\in \N$.

\begin{remark}\label{pmeje}
If we intersect $S$ with any axis, the set obtained is formed by the elements of $\N^p$ fulfilling a proportionally modular Diophantine inequality or it is equal to $\{0\}$. Thus, this intersection is isomorphic to a proportionally modular numerical semigroup.	
Besides, every $x\in \N^p$ satisfying $g(x)\geq b$ belongs to $S$.
\end{remark}

\begin{remark}
Let $w=(w_1,\dots,w_p)\in\N^p\setminus\{0\}$ such that $\gcd(w_1,\dots,w_p)=1$. The set $\{\lambda w|\lambda \in \Q\}\cap \N^p$ is equal to $\{x w~|~x\in \N\}$. If $g(w)>0$, then $g(xw)>0$ for all $x\in \N\setminus\{0\}$. Since $g(x_1,\dots,x_p)=g_1x_1+\dots+g_px_p$
and $f(x_1,\dots,x_p)=f_1x_1+\dots+f_px_p$ with $f_i,g_i\in\Z$ for all $i\in[p]$, then
$g(xw)=g(xw_1,\dots,xw_p)=(g_1w_1+\dots+g_pw_p)x$ and
$f(xw)=f(xw_1,\dots,xw_p)=(f_1w_1+\dots+f_pw_p)x$.
The element $c'=g_1w_1+\dots+g_pw_p$ is in $\N$ and $a'=f_1w_1+\dots+f_pw_p$ belongs to $\Z$.
So, $xw\in S$ if and only if $ {a'}  x\mod b\leq  {c'}  x$.
Hence, the submonoid $\{x w~|~x\in \N\}\cap S$ is isomorphic to a proportionally modular numerical semigroup.
	
We suppose now that $g(w)=0$. Let $u$ be the nonnull element of $\{\lambda w|\lambda \in \Q\}\cap \N^p$ closest to the origin verifying that $f(u)\mod b\leq g(u)$. By the linearity of $g$ we have $g(u)=0$
and thus $f(u)\mod b=0$.
Assume that $w'\in \{\lambda w|\lambda \in \Q\}\cap \N^p$ and $f(w')\mod b\leq g(w')$.
If $w'\not\in \{k u~|~k\in\N\}$, consider $k'\in\N$ such that $\|k'u\|<\|w'\|$ and $\|(k'+1)u\|>\|w'\|$.
The element $w'-k'u\in\N^p$ verifies $g(w'-k'u)=0$, $f(w'-k'u)\mod b=0\leq g(w'-k'u)$ and
$\|w'-k'u\|<\|u\|$, which is a contradiction. Thus, $\{\lambda w~|~\lambda \in \Q,~f(\lambda w)\mod b\leq g(\lambda w)\}\cap \N^p$ is equal to
$\{\lambda u~|~\lambda \in \N\}$ and this submonoid is isomorphic to $\N$.

The above two paragraphs allow us to view  proportionally modular semigroups  as beams of proportionally modular numerical semigroups.
\end{remark}

In the following result we give an effective proof that a proportionally modular semigroup is finitely generated. This proof is tailored to these semigroups.

\begin{theorem}\label{thfg}
Every proportionally modular monoid of $\N^p$ is finitely generated.
\end{theorem}
\begin{proof}
We suppose again that $g(x_1,\dots,x_p)=g_1x_1+\dots+g_px_p$ with $g_1,\dots,g_p\in \Z$, and denote the $p$-tuple $(x_1,\dots,x_p)$ by $x$.
Let $S$ be the proportionally modular monoid defined by the inequality $f(x)\mod b\leq g(x)$.

We have two straightforward cases: $g_1,\dots,g_p< 0$ and $g_1,\dots,g_p > 0$.
If $g_1,\dots,g_p< 0$, then $S=\{0\}$.
If $g_1,\dots,g_p > 0$, every  set $\{x\in \N^p~|~g(x)=i\}$ with $i\in \N$ is finite.
Since $\N^p\setminus S\subset \cup_{i=0}^{b-1}\{x\in \N^p~|~g(x)=i\}$, the
set $\N^p\setminus S$ is finite, and  hence a minimal generating set of $S$ can be computed by using \cite[Corollary 9]{MR3110598}.

Assume that there exist $i,j\in [p]$ such that $g_ig_j\le 0.$ Let $U=\{\u_1,\dots,\u_t\}$ be the minimal generating set of the $\N$-solutions of the system of
Diophantine equations (see \cite{Nsolutions})
\[
\left\{
\begin{array}{l}
f(x)\mod b=0,\\
g(x)=0.
\end{array}
\right.
\]
Every $x\in\N^p$ verifying $g(x)\geq b$ is in $S$
and therefore
$S\cap \{x\in\N^p~|~g(x)\geq b\}$ is equal to
$\{x\in\N^p~|~g(x)\geq b\}$ (see Remark \ref{pmeje}).
So, we obtain that
\[\begin{multlined}
S\setminus \cup_{i=1}^{b-1}\{ x\in \N^p ~|~ g(x)=i\}=\\
S\cap (\{x\in\N^p~|~g(x)=0\}\cup \{x\in\N^p~|~g(x)\geq b\})=\\
(S\cap \{x\in\N^p~|~g(x)=0\}) \cup \{x\in\N^p~|~g(x)\geq b\}.
\end{multlined}\]
	
Take $x\in S\cap ( \cup_{i=1}^{b-1}\{ x\in \N^p ~|~ g(x)=i\})$ and
assume that $g(x)=d$ with $d\in\{1,\dots,b-1\}$. This element is a $\N$-solution
of a Diophantine system of equations of the form
\begin{equation}\label{eqk}
\left\{
\begin{array}{l}
f(x)\mod b=k,\\
g(x)=d.
\end{array}
\right.
\end{equation}
with $k=0,\dots,d$.
Let $M_{dk}$ be the set of minimal $\N$-solutions of (\ref{eqk}).
By \cite{Nsolutions}, the element
$x$ can be expressed as $m+\sum_{i=1}^t\lambda_i \u_i$ with $m\in M_{dk}$ and $\lambda_i\in \N.$ Figure \ref{figure_3D} illustrates graphically the sets $M_{dk}$ of an example.
\begin{figure}[h]
\begin{center}
\begin{tabular}{|c|}\hline
\includegraphics[scale=.65]{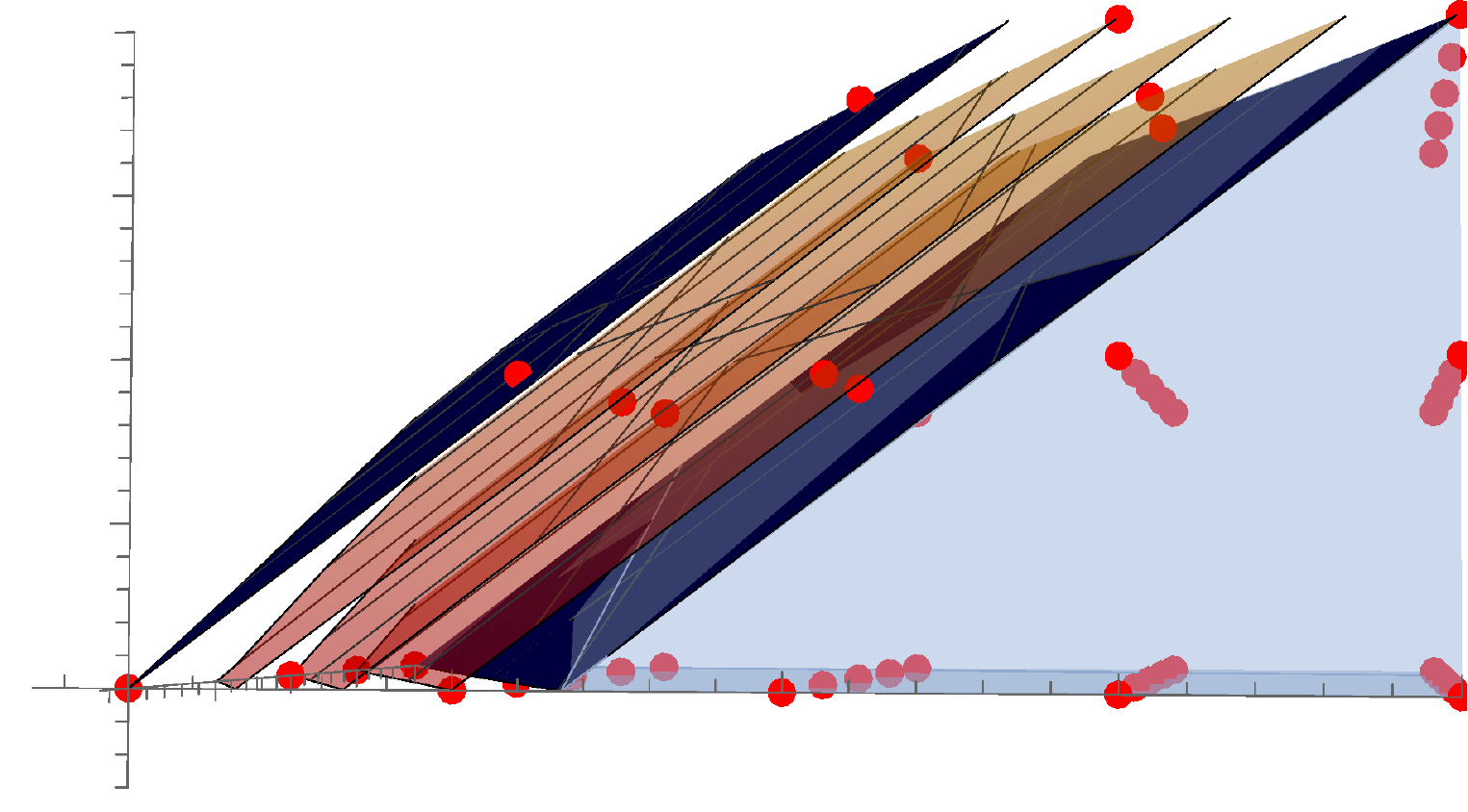}\\
\hline
\end{tabular}
\caption{Semigroup given by the inequality $5x+2y+z\mod 4 \le 3x+y-4z.$}\label{figure_3D}
\end{center}
\end{figure}

We now construct a generating set of the monoid
$$\{x\in\N^p~|~g(x)\geq b\}\cup (S\cap \{x\in\N^p~|~g(x)=0\}).$$
An graphical example of this monoid is showed in Figure \ref{figure_3D}.
Let $V=\{v_1,\ldots ,v_l\}$ be the minimal generating set of $\{x\in\N^p~|~g(x)=0\}.$ Note that $bv_i\in \langle U\rangle$ for all $i=1,\ldots ,l.$
Since the set ${\rm L}(S)\cap \N^p$ is the  cone determined by the hyperplanes
$g(x)\geq 0$, $x_1\geq 0,\dots,x_p\geq 0$, it is a
finitely generated monoid (see \cite[Section \S 7.2 and Theorem 16.4]{MR874114}).
Define $$\FraC_0=V\cup \{w_1^1,\ldots ,w_{n_1}^1,w_1^2,\ldots ,w_{n_2}^2,\ldots ,w_1^{b-1},\ldots ,w_{n_{b-1}}^{b-1},\widetilde{w}_1,\ldots ,\widetilde{w}_q\}$$
the minimal generating set of ${\rm L}(S)\cap \N^p$ satisfying:
\begin{itemize}
\item $g(w^i_j)=i$ for all  $i=1,\ldots ,b-1,$ and $j=1,\ldots ,n_i.$
\item $g(\widetilde{w}_j)\ge b$ for all  $j=1,\ldots ,q.$
\end{itemize}
Let  $\FraC_1$ be the finite set
$$(\FraC_0\setminus\{w_1^1,\ldots ,w_{n_1}^1\}) \cup \cup_{i\in [n_1]} \{2w_i^1,3w_i^1\}
\cup \cup _{j\in [n_1]} \{w_j^1+s|s\in \FraC_0\setminus V\}.
$$
Note that for every element $s\in \FraC_1,$ $g(s)=0$ or $g(s)\ge 2.$
Besides, if $s\in ({\rm L}(S)\cap \N^p)\setminus \{x\in\N^p~|~g(x)=1\}$
we also have that $g(s)=0$ or $g(s)\ge 2.$
If $g(s)=0$, then $s$ belongs to the semigroup generated by $V\subset \FraC_1.$
If $g(s)\ge 2,$ we consider $\lambda_j,\nu_{ij}, \mu_j\in \N$ such that
$$s=\sum _{j=1}^l \lambda _jv_j + \sum_{j=1}^{n_1} \nu _{1j}w^1_j +\sum_{i=2}^{b-1} \sum_{j=1}^{n_i} \nu _{ij}w^i_j + \sum_{j=1}^q \mu _j \widetilde{w}_j.$$
For each $\nu_{1j}>1,$ the addend $\nu_{1j}w^1_j$
can be replace by a non-negative integer linear combination of the elements of  $\{2w^1_j,3w^1_j\}\subset \FraC_1.$ Since $g(s)\ge 2,$
if there exists $\nu_{1j}=1,$ then there exists a nonnull coefficient $\nu_{i'j'}$
with $j'\neq j$ and/or there exists a nonnull coefficient $\mu_{j''}$.
Hence, $w^1_j+\nu_{i'j'}w^{i'}_{j'}$ or $w^1_j+\mu_{j''}\widetilde{w}_{j''}$
appears in the expression of $s$.
Note that both elements are obtained as a non-negative integer linear combination of elements of $\FraC_1.$
In any case, the element $s$ is in the semigroup generated by $\FraC_1,$
and therefore $\FraC_1$ is a system of generators of $({\rm L}(S)\cap \N^p)\setminus \{x\in\N^p~|~g(x)=1\}.$

Once we have the sets $\FraC_0$ and $\FraC_1$, the sets
$\FraC_k$ with  $k\in [b-1]$ are obtained recursively as follows:
$$\FraC_k=(\FraC_{k-1}\setminus\{w_1^k,\ldots ,w_{n_k}^k\}) \cup \cup_{i\in [n_k]} \{2w_i^k, 3w_i^k\}
\cup \cup _{j\in [n_k]} \{w_j^k+s|s\in \FraC_{k-1}\setminus V\}
.$$
Reasoning as above, it is straightforward to prove that the set  $\FraC_k$
is a system of generators of the semigroup $({\rm L}(S)\cap \N^p)\setminus \cup _{j=1}^k \{x\in\N^p~|~g(x)=j\}.$
Thus, $\FraC_{b-1}$ is a system of generators of  $\{x\in\N^p~|~g(x)\geq b\}\cup \{x\in\N^p~|~g(x)=0\}.$

Let $\FraC$ be the finite set $\FraC=(\FraC_{b-1}\setminus V)\cup U.$ Then, $S\cap \{x\in\N^p~|~g(x)=0\}=\langle \FraC\rangle\cap \{x\in\N^p~|~g(x)=0\}.$
Let $\widetilde{\FraC}$ be the finite set
$$\FraC\cup \bigcup _{w\in \FraC\cap \{x\in\N^p~|~g(x)\ge b\}}
\{z\in \N^p| w\prec z\preceq  w+\sum_{i=1}^lbv_i\},
$$
and let $s\in  \{x\in\N^p~|~g(x)\geq b\}\cup (S\cap \{x\in\N^p~|~g(x)=0\}).$
If $g(s)=0,$ $s$ can be obtained from the set $U\subset \widetilde{\FraC}.$ Otherwise, $$s=\sum _{\widehat{w}\in \FraC,g(\widehat{w})\ge b}\lambda_{\widehat{w}} \widehat{w} +\sum _{i=1}^l\mu _i v_i,
$$
with $\lambda _{\widehat{w}},\mu_i\in \N.$ For every $\mu_i,$ let $\mu_i',r_i\in \N$ be
the quotient and the remainder of the integer division $\mu_i/b$ ($\mu_i=\mu_i'b+r_i$ with $r_i\in [b-1]$).
With these values we have
$$s=\sum _{\widehat{w}\in \FraC,g(\widehat{w})\ge b}\lambda _{\widehat{w}} \widehat{w} +\sum _{i=1}^l\mu _i'b v_i +\sum _{i=1}^l r_i v_i,
$$
where $\sum _{i=1}^l\mu _i'b v_i\in \langle U\rangle\subset \langle \widetilde{\FraC}\rangle$ and $\sum _{\widehat{w}\in \FraC,g(\widehat{w})\ge b}\lambda _{\widehat{w}} \widehat{w}_i +\sum _{i=1}^l r_i v_i$
belong to $\langle\widetilde{\FraC}\rangle.$
Hence, $\widetilde{\FraC}$ is a system of generators of  $\{x\in\N^p~|~g(x)\geq b\}\cup (S\cap \{x\in\N^p~|~g(x)=0\})$, and thus
 a system of generators of $S$ is the finite set
$\widetilde{\FraC}\cup (\cup_{d=1}^{b-1}\cup_{k=0}^d M_{dk} )$.
\end{proof}

From the above proof we obtain an algorithm to compute systems of generators of proportionally modular affine semigroups.
Although, with this method it is necessary to solve several system of Diophantine equations
(this is a $\mathcal{NP}$-complete problem) and  consider several big sets of elements.
A particular case where we avoid these issues is when
$p=2$. This case is studied in next sections. We describe a geometrical approach, easier to solve, that allows us to determinate the Cohen-Macaulayness, Gorensteiness and Buchsbaumness.

\section{Proportionally modular affine semigroups of $\N^2$}\label{n2}

This section is about proportional modular semigroups associated to modular Diophantine inequalities into two variables: $f(x,y)\mod b\le g(x,y)$ where $f(x,y)=f_1x+f_2y,$ $g(x,y)=g_1x+g_2y$ with  $b\in \N$ and $f_1,f_2, g_1, g_2\in\Z$. As in previous sections, we denote by $S$ the proportionally modular affine semigroup associated to the above modular inequality. In this section, we provide a geometrical algorithm to compute their minimal generating sets. Besides, their associated (minimal) Frobenius vectors are studied.

Given a subset $A\subset \Q^2$,
we denote by $\Tint ({\rm L}(A))$ the topological interior of the cone ${\rm L}(A)$. This set is equal
to $\Tint(A)=\{ \sum_{i=1}^m \lambda_i a_i~|~ \lambda_i\in \Q_>,~a_i\in A,~m\in \N\}$.
Note that given a subsemigroup $S$ of $\N^2$ minimally generated by $\{s_1,\dots,s_p\}$ there exists a minimal set of elements $\{s_{i_1},\dots,s_{i_t}\}$ such that its associated cone ${\rm L}(S)$ is equal to $\{ \sum_{j=1}^t \lambda_j s_{i_j}~|~ \lambda_j\in \Q_\geq\}.$ The semigroup $S$ is called simplicial whenever ${\rm L}(S)={\rm L}(\{s_1,s_2\}).$ Note that every nontrivial proportionally modular semigroup of $\N^2$ is simplicial.

\begin{definition}\label{vectorU}
Assume $g(x,y)=g_1x+g_2y$ with $g_1g_2\leq 0$.
Denote by $\u$ the generator of the semigroup given by the $\N$-solutions of
\[
\left\{
\begin{array}{l}
g_1x+g_2y=0,\\
f_1x+f_2y\mod b=0.
\end{array}
\right.
\]
If $g_1\neq 0$, this element is the minimal solution of $(\frac {-g_2f_1+f_2g_1} {g_1} y \mod b=0)$.
Analogously, if $g_2\neq 0$ we also obtain only an element. In particular, note that $\u$ is a minimal generator of $S.$
\end{definition}

The vector $\u$ has a nice property which allows to obtain all the elements belonging to $S$ from a strip of $S.$

\begin{lemma}\label{traslacionU}
	Assume $g(x,y)=g_1x+g_2y$ with $g_1g_2\leq 0$, let $\u\in\N^2$ be as in Definition \ref{vectorU} and
	$v,w\in\N^2$ such that $v+\u=w$.
	Then, $v\in S$ if and only if $w\in S$.
\end{lemma}
\begin{proof}
Assume $v\in S$. We have $f(v)\mod b\leq g(v)$ and $f(\u)\mod b=g(\u)=0$. Thus,
$f(w)\mod b=f(v+\u)\mod b=(f(v)+f(\u))\mod b \leq f(v)\mod b+f(\u)\mod b\leq g(v)=g(v)+g(\u)=g(v+\u)=g(w)$.

If $w\in S$, we can proceed similarly to obtain that $v\in S$.
\end{proof}

The following result provides us with an alternative geometrical and  effective proof of Theorem \ref{thfg} for a subsemigroup on $\N^2$.

\begin{proposition}\label{fgN2}
Every nontrivial proportionally modular semigroup $S$ of $\N^2$ is finitely generated.
\end{proposition}
\begin{proof}

We have $f,g:\Q^2\to \Q$. Assume that $g(x,y)=g_1x+g_2y$ with $g_1,g_2\in\Z$.
If $g_1,g_2< 0$, then $S=\{(0,0)\}$.
For the other cases we distinguish
two main cases: $g_1,g_2> 0$ and $g_1g_2\le 0.$

Assume that $g_1,g_2> 0$.  The elements of $\N^2$ not belonging to $S$ are in the straight lines
$g(x,y)=1,\dots, g(x,y)=b-1$. Furthermore, the intersection of every straight line $g(x,y)=d$ with
$\Q_\geq^2$ is the segment with endpoints $(0,\frac d {g_2})$, $(\frac d {g_1}, 0)$. Thus, in the straight
lines $g(x,y)=1,\dots, g(x,y)=b-1$ there are only a finite number of elements in $\N^2$.
Hence, $\N^2\setminus S$ is finite and the minimal generating set $S$ can be obtained from $\N^2\setminus S$ using \cite[Corollary 9]{MR3110598}.

Assume now that $g_1g_2\le 0.$ Let $\u=(u_1,u_2)$ be as in Definition \ref{vectorU}.
The intersection of $g(x)=b$ with one of the two axes is not empty.
Assume that this axis is the $x$-axis and then $g_1>0$ and $g_2\le 0$.
Let $w=(\frac {b} {g_1},0)$ be the point of intersection of the straight line $g(x,y)=b$ with the line $y=0,$ and $\widetilde{\u}=(\widetilde{u}_1,0)\in\N^2$ be the minimal generator of the affine semigroup $S\cap OX$ closest to the origin.
By Remark \ref{pmeje}, the element $\widetilde{\u}$ exists  and it is a minimal generator of $S.$
Fixed $s=(s_1,s_2)\in S,$ by Lemma \ref{traslacionU}, if $s_2\ge u_2,$ $s$ satisfies $s-\u\in S,$ and it can be obtained from an element $s'=(s_1',s_2')\in S$ with $s_2'<u_2$ adding $\lambda \u$ for some $\lambda\in \N.$
In case $s$ belongs to the half-space $\{(x,y)\in \R_+^2| g(x,y)\ge b\}+\widetilde{\u},$
the element $s-\widetilde{\u}$ belongs to $S.$ Let $\G$ be the finite set given by $S\cap (\R^2\setminus ((\{(x,y)\in \R^2| g(x,y)\ge b\}+\widetilde{\u}) \cup \{(x,y)\in R^2| y>u_2\})).$ This set is equal to  $S\cap \text{ConvexHull}(\{O,\u,\u+w+\widetilde{\u},w+\widetilde{\u}\})$ and $S\setminus \G$ is equal to $\{h+\lambda_1\u+\lambda_2\widetilde{\u} |h\in \G\mbox{ and }\lambda_1,\lambda_2\in \N\}.$ Then, $S$ is finitely generated by the elements belonging to $\G.$
Analogously, the case $g_1\le0$ and $g_2> 0$ can be solved considering $w=(0,\frac b {g_2}),$ and $\widetilde{\u}=(0,\widetilde{u}_2)\in S$ the minimal generator of the affine semigroup $S\cap OY$ closets to the origin.
\end{proof}

From now on, we denote by $w$ the element $\{x\in \R_+^2|g(x)=b\} \cap (OX\cup OY)$ when $g_1g_2\le 0,$ by $\G$ the set $S\cap \text{ConvexHull}(\{O,\u,\u+w+\widetilde{\u},w+\widetilde{\u}\}),$  and by $\widetilde{\u}$ the vector $(\widetilde{u}_1,0)$ or $(0,\widetilde{u}_2)$ as in above Proposition.

In Algorithm \ref{algoritmo} we formulate a faster algorithm to compute the minimal generating set of a proportionally modular semigroup in $\N^2.$
\begin{algorithm}
\caption{Computation of the minimal generating set of a proportionally modular affine semigroup $S$.}\label{algoritmo}
\textbf{Input:} The proportionally modular Diophantine inequality $f(x)\mod b\leq g(x)=g_1x+g_2y.$\\
\textbf{Output:} The minimal generating set of $S.$
\begin{algorithmic}[1]
\If {$g_1,g_2<0$} \Return $\{(0,0)\}$. \EndIf
\If {$g_1,g_2> 0$} compute the finite set $\N^2\setminus S.$ The minimal generating set $H$ of $S$ can be obtained from $\N^2\setminus S$ using \cite[Corollary 9]{MR3110598}.
\State \Return $H$.
\EndIf
\If {$g_1g_2\le 0$}
\State Compute the vector $\u$ defined in Definition \ref{vectorU}.
\If {$g_1\ge 0$} $\widetilde S:=\{(x,0)~|~f(x,0)\mod b\le g(x,0)\}$.
\EndIf
\If {$g_1< 0$} $\widetilde S:=\{(0,y)~|~f(0,y)\mod b\le g(0,y)\}$.
\EndIf
\State Compute the minimum minimal generator $\widetilde{\u}$ of the subsemigroup $\widetilde S$.
\State $w:= \{x\in \R_+^2|g(x)=b\} \cap (OX\cup OY).$
\State $\G:=S\cap \text{ConvexHull}(\{O,\u,\u+w+\widetilde{\u},w+\widetilde{\u}\})$.
\State Obtain $H$ a minimal system of generators from  $\G$.
\State \Return $H.$
\EndIf
\end{algorithmic}
\end{algorithm}
In this algorithm, for the case $g_1,g_2>0$ there exists an alternative way to compute the step 2. Let $w_1$ and $w_2$ be the intersections of $\{x\in\Q^2|g(x)=b\}$ with $OX$ and $OY$ respectively, and let $\widetilde{\u}_1$ and $\widetilde{\u}_2$ be the minimum minimal generators of the semigroups $S\cap OX$ and $S\cap OY$ respectively (note $\widetilde{\u}_1$ and $\widetilde{\u}_2$ are in the minimal generating set of $S$). Consider $\CaJ$ the set $\text{ConvexHull} (\{ O, w_1+\widetilde{\u}_1, w_2+\widetilde{\u}_2\}).$ It is straightforward to prove that $\N^2\setminus \CaJ\subset S$ and every $x\in \N^2\setminus \CaJ$ satisfies that $x- \widetilde{\u}_1$ and/or $x-\widetilde{\u}_2$ are/is in $S.$ Thus, the elements belonging to $\N^2\setminus \CaJ$ can be obtained from elements in $\CaJ\cap S$ and then a system of generators of $S$ is included in $\CaJ\cap S.$ So, an alternative way to do the step 2 is to compute $w_1,w_2,\widetilde{\u}_1,\widetilde{\u}_2,$ consider the set $\CaJ\cap S$ and take $H$ the minimal generating set of $S$ from $\CaJ\cap S.$

\begin{example}
Let $f(x,y)= 3x-2y,$ $g(x,y)= x-3y$, $b= 11$
and $S$ the proportionally modular affine semigroup defined by the modular inequality $f(x,y)\mod b\le g(x,y).$
In order to obtain a generating set of $S$, we follow the steps of Algorithm \ref{algoritmo}.
First, we compute the minimal non null vector $\u\in \N^2$ solving the system of modular equations $\{f(x,y)\mod b\le g(x,y),g(x,y)=0\}.$ In this case we have $7y \mod 11 =0.$ So, $\u=(33,11).$ For computing $\widetilde{\u}$ it is needed to solve the modular equation $f(x,0)\mod b\le g(x,0)\equiv 3x \mod 11 \le x$. The proportionally numerical semigroup obtained is the generated by $\{4,5,11\},$ and therefore $\widetilde{\u}=(4,0).$ The point $w=\{(x,y)\in \R^2|g(x,y)=11\} \cap OX$ is equal to $w=(11,0).$
Another set we need is
$\G=S\cap \text{ConvexHull}(\{(0,0),(33,11),(48,11),(17,0)\}).$
The last step is performed taking the minimal elements of $\G$.

The software \cite{ProporcionallyModularAffineSemigroupN2} allows us to
compute the minimal generating set of $S$ just as follows:
\begin{verbatim}
In[1]:= ProporcionallyModularAffineSemigroupN2[3, -2, 11, 1, -3]

Out[1]= {{4., 0.}, {5., 0.}, {5., 1.}, {8., 1.}, {9., 2.},
            {11., 0.}, {13., 3.}, {14., 4.}, {18., 5.},
            {19., 6.}, {23., 7.}, {28., 9.}, {33., 11.}}
\end{verbatim}

\end{example}

To conclude this section we make a geometric approach to the computation of the Frobenius vectors of a proportionally modular affine semigroup of $\N^2.$
By using this approximation, we present an algorithm to determine the minimal Frobenius vectors in these semigroups.

\begin{definition}
We say that an affine semigroup $T$ has a Frobenius vector if there exists an element $q\notin T$ belonging to the group $G(T)$ (the subgroup of $\N^p$ generated by $T$) such that $(q+\Tint ({\rm L} (T))) \cap G(T)\subset S\setminus \{0\}.$ A Frobenius vector is called minimal Frobenius vector if it is minimal with respect to the product ordering on $\N^p$.
\end{definition}

\begin{proposition}
Let $S\subset \N^2$ be a nontrivial proportionally modular semigroup. Then the following hold:
\begin{itemize}
\item If $g_1g_2\le 0,$ the unique minimal Frobenius vector of $S$ is the minimal integer element of $\text{ConvexHull} (\{O,\u,w,w+\u\})\setminus S$ closest to the line $\{x\in\R^2| g(x)=b\}.$
\item If $g_1g_2>0,$  the (minimal) Frobenius vectors are in the finite set $(\text{ConvexHull}(\{O,w_1,w_2\})\cap \N^2)\setminus S.$
\end{itemize}
\end{proposition}

\begin{proof}
We consider two cases: $g_1g_2>0,$ and $g_1g_2\le 0.$

Assume that $g_1g_2>0,$ and let $w_1$ and $w_2$ be the unique two elements of  $\{x\in \R^2|g(x)=b\} \cap (OX\cup OY).$ In this case, note that ${\rm L}(S)$ is equal to $\Q^2_+$ and $({\rm L}(S)\cap \N^2)\setminus S$ is the nonempty finite set $\Delta=(\N^2\cap \text{ConvexHull}(\{0,w_1,w_2\}))\setminus S.$ Any maximal element $\omega$ in $\Delta$ satisfies that $(\omega+\Tint ({\rm L} (S))) \cap G(S) \subset S.$ So, these maximal elements are Frobenius vectors. Besides, for any non maximal element $\omega _1$ belonging to $\Delta,$ there exists $\omega\in \Delta$ maximal such that $\omega \in \omega_1 + {\rm L}(S)$ but it is possible that there is not a maximal element belonging to $\omega_1 + \Tint( {\rm L}(S)).$ In that case, $\omega_1$ is also a Frobenius vector.
Thus, every minimal Frobenius vector is a maximal elements in $\Delta$ or an element $\omega_1$ in $\Delta$ such that there is no maximal element belonging to $\Delta$ in $\omega_1 + \Tint( {\rm L}(S)).$

Assume now $g_1g_2\le 0.$ The cone ${\rm L}(S)$ is generated by $\u$ and $w,$ where $w$ is the point $\{x\in \R_+^2|g(x)=b\} \cap (OX\cup OY).$ Since $\u$ is a non-negative vector, any integer element $q\in {\rm L}(S)\setminus (w+{\rm L}(S))$ can be expressed as $q=p+\lambda \u$ with $\lambda \in \N$ and $p\in \N^2\cap \text{ConvexHull}(\{O,\u,w,w+\u\}).$ Thus, the minimal Frobenius vectors of $S$ belong to the finite and nonempty set $\Delta = (\N^2\cap \text{ConvexHull}(\{O,\u,w,w+\u\}))\setminus S.$ Let $\w$ be the minimal closest point to the line $\{g(x)=b\}$ belonging to $\Delta.$ We have $(\w + \Tint({\rm L}(S)))\cap \N^2 \subset S.$ Besides, for any other integer element $q\in ({\rm L}(S)\setminus S)\setminus (\w +{\rm L}(S))$ there exists $\lambda \in \N$ such that $\w +\lambda \u\notin S$ and  $\w +\lambda \u \in q+{\rm L}(S).$ Thus $q$ is not a minimal Frobenius vector and we conclude that $S$ has a unique minimal Frobenius vector which is the minimal of the closest elements to $\{g(x)=b\}$ belonging to $\Delta.$
\end{proof}

An algorithm to compute the minimal Frobenius vectors can be formulated from the above Proposition.

\begin{example}
In this example we illustrate the concept of minimal Frobenius vector. Let $f(x,y)= 3x+2y,$ $g(x,y)= x-y$ and $b= 10$ be the elements of the modular inequality $f(x,y)\mod b\le g(x,y)$ and $S$ be its proportionally modular semigroup associated. The point $\u$ is $(2,2)$ and $w=(10,0).$ Figure \ref{Frobenius_vector_fig1} illustrates the situation.
\begin{figure}[h]
\begin{center}
\begin{tabular}{|c|}\hline
\includegraphics[scale=.65]{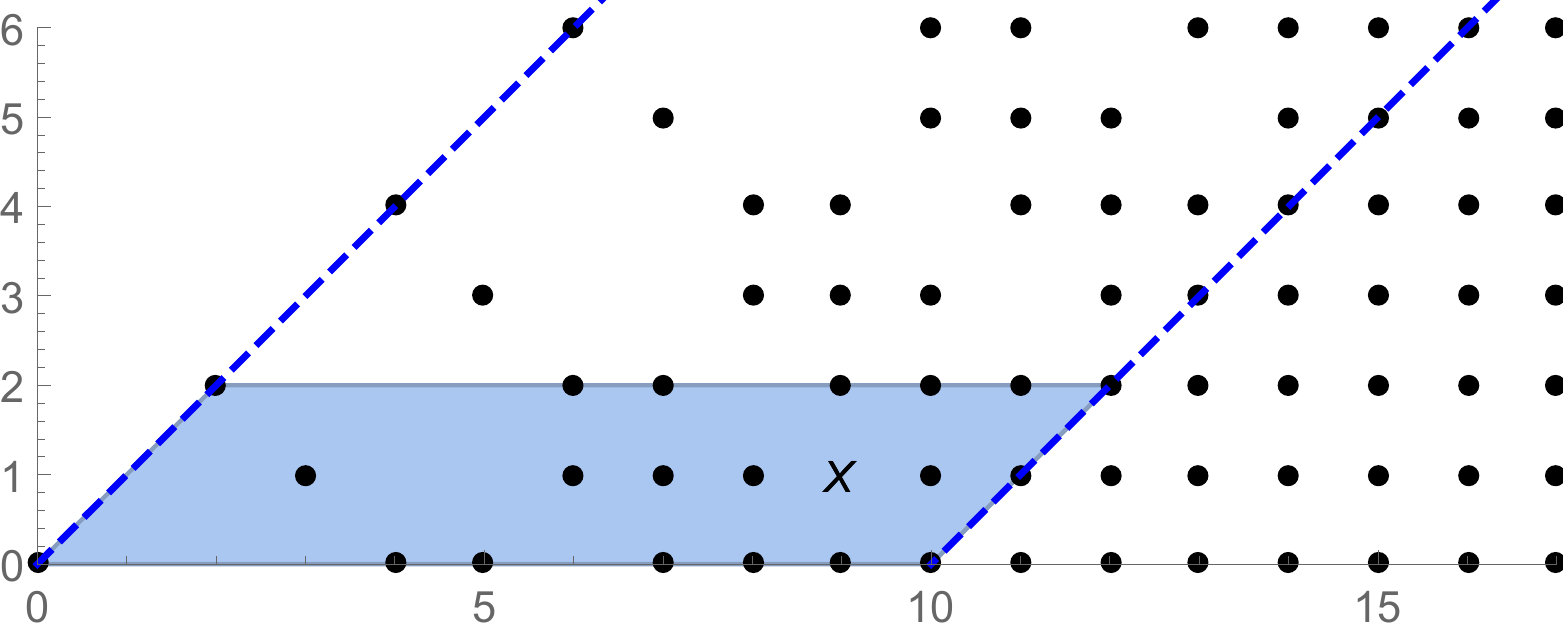}\\
\hline
\end{tabular}
\caption{Example of minimal Frobenius vector.}\label{Frobenius_vector_fig1}
\end{center}
\end{figure}
The black points are the points in $S,$ the dashed lines are $\{g(x,y)=0\}$ and $\{g(x,y)=b\}$ respectively, the shady region is the set $\text{ConvexHull}(\{O,\u,w,w+\u\}),$ and the point $\w=(9,1)$ is the unique minimal Frobenius vector. Note that $\w+\lambda \u,$ with $\lambda \in \N,$ is again a Frobenius vector.
\end{example}

\section{Some properties of proportionally modular semigroups of $\N^2.$}\label{propiedades_n2}

In this section we study the Cohen-Macaulayness, Gorensteinness and Buchsbaumness of proportionally modular affine semigroups of $\N^2.$ For this objective we consider only simplicial proportionally modular semigroups, and this occurs if the coefficients $g_1$ and $g_2$ of $g(x,y)=g_1x+g_2y$ are not both lesser than or equal to zero. Denote by $S$ the proportionally modular semigroup of the $\N$-solutions of $f(x,y)\mod b\le g(x,y),$ and by $\G$ the set defined after Proposition \ref{fgN2}. Again, $\u$ is the vector defined in Definition \ref{vectorU} and $\widetilde{\u}$ is the vector that appears after Proposition \ref{fgN2}.

The following result characterizes Cohen-Macaulay simplicial affine semigroups of $\N^2.$

\begin{proposition}(\cite[Corollary 2]{MR3212609})\label{C-M}
Let $T\subseteq \N^2$ be an affine simplicial semigroup,
the following conditions are equivalent:
\begin{enumerate}
\item $T$ is Cohen-Macaulay.
\item For all $v\in ({\rm L}(T)\cap \N^2)\setminus T$, $v+s_1$ or $v+s_2$ does not belong to $T$ where $s_1$ and $s_2$ are minimal generators of $T$ such that ${\rm L}(T)=\langle s_1,s_2\rangle.$
\end{enumerate}
\end{proposition}

As in the proof Proposition \ref{fgN2}, if $g_1$ and $g_2$ satisfies $g_1g_2>0,$ $\N^2\setminus S$ is finite. So, $S$ is simplicial but does not satisfy the second condition of Proposition \ref{C-M}, and thus $S$ is not Cohen-Macaulay. For this reason, in what follows we assume that $g_1g_2\le 0.$ Besides, we assume that the trivial case $S=\N^2$ does not happen. Fixed these conditions, next result characterizes the Cohen-Macaulay proportionally modular semigroup.

\begin{corollary}\label{proporcional_C-M}
Any proportionally modular semigroup $S$ is Cohen-Macaulay.
\end{corollary}
\begin{proof}
With the fixed conditions, the semigroup $S$ is an affine simplicial semigroup (Proposition \ref{fgN2}) and the vector $\u$ can be considered as one of the minimal generator $s_1$ or $s_2$ appearing in Proposition \ref{C-M}.
By Lemma \ref{traslacionU}, for any $v\in ({\rm L}(S)\cap \N^2)\setminus S$ it is verified that $v+\u\notin S.$ So, $S$ is Cohen-Macaulay.
\end{proof}

We focus now our attention on the Gorenstein property. We characterize this property in terms of the intersection of the
Ap\'{e}ry sets of some minimal generators belonging to the extremal rays of its associated cone. Recall that the Ap\'{e}ry set associated to an element $s$ in any semigroup $T$ is the set $\ap(s)=\{a\in T|a-s\notin T\}.$
The following result appears in \cite[Theorem 4.6]{RosalesCM}.

\begin{theorem}\label{Gorenstein_rosales}
For a given affine simplicial semigroup $T,$ the following conditions are equivalent:
\begin{enumerate}
\item $T$ is Gorenstein.
\item $T$ is Cohen-Macaulay and $\cap_{i=1}^2 \ap(s_i)$ has a unique maximal element (with respect to the order defined by $T$) where where $s_1$ and $s_2$ are minimal generators of $T$ such that ${\rm L}(T)=\langle s_1,s_2\rangle$.
\end{enumerate}
\end{theorem}

Since $\u$ and $\widetilde{\u}$ are minimal generators of $S$ and ${\rm L}(S)=\langle \u, \widetilde{\u}\rangle,$ the study of the set $\ap(\u)\cap\ap(\widetilde{\u})$ allows us to check whether a proportional modular semigroup is Gorenstein or not.

\begin{lemma}
Let $S$ be a proportional modular semigroup verifying the fixed conditions. The set $\ap(\u) \cap \ap(\widetilde{\u})$ is equal to the finite set $\{h\in \G | h-\u,h-\widetilde{\u}\notin S\}.$
\end{lemma}
\begin{proof}
Let $s$ be an element belonging to $S\setminus \G.$ Using a similar argument of the proof of Proposition \ref{fgN2}, we obtain that $S\setminus \G$ is the set $\{h+\lambda_1\u+\lambda_2\widetilde{\u} |h\in \G\mbox{ and }\lambda_1,\lambda_2\in \N\}.$ Note that the vectors $\u$ and $\widetilde{\u}$ can be considered as the minimal generators $s_1$ and $s_2$ of Theorem \ref{Gorenstein_rosales}, and then $(\cap_{i=1}^2 \ap(s_i))\cap (S\setminus \G)=\emptyset.$ Therefore, $\ap(s_1) \cap \ap(s_2)=\{h\in \G | h-s_1,h-s_2\notin S\}.$
\end{proof}

\begin{corollary}
Let $S$ be a proportional modular semigroup verifying the fixed conditions. The semigroup $S$ is Gorenstein iff there exists a unique maximal element belonging to $\{h\in \G | h-\u,h-\widetilde{\u}\notin S\}.$
\end{corollary}

For the last result of this work we have to define a semigroup associated to $S.$ In general, given an affine semigroup $T$ minimally generated by $\{s_1,\ldots ,s_t\}$, denote by $\overline T$ the affine semigroup $\{s\in \N^p| s+s_i\in T,\, \forall i=1,\ldots ,t\}$.
In \cite{RosalesBuchs}, it is given a characterization of Buchsbaum simplicial semigroups $T$ in terms of their semigroups $\overline T$.

\begin{theorem}\cite[Theorem 5]{RosalesBuchs}\label{RosalesBuchs}
The following conditions are equivalent:
\begin{enumerate}
\item $T$ is an affine Buchsbaum simplicial semigroup.
\item $\overline{T}$ is Cohen-Macaulay.
\end{enumerate}
\end{theorem}

\begin{corollary}
Let $S$ be a proportional modular semigroup verifying the fixed conditions. Then, $S$ is Buchsbaum.
\end{corollary}
\begin{proof}
Note that the element $\u\in S$ is a minimal generator of $S,$ and that for every $a\in {\rm L}(S)\setminus S,$ $a+\u\notin S.$ Thus, $S=\overline{S}.$
Since $S$ is Cohen-Macaulay (see Corollary \ref{proporcional_C-M}), by Theorem \ref{RosalesBuchs}, $\overline{S}$ is Cohen-Macaulay.
\end{proof}

\begin{example}
Consider the modular inequality $7x-y \mod 5 \le x-14y.$ Its minimal generating set is:
\begin{verbatim}
In[1]:= ProporcionallyModularAffineSemigroupN2[7, -1, 5, 1, -14]

Out[1]= {{3., 0.}, {4., 0.}, {5., 0.}, {16., 1.}, {17., 1.},
            {18., 1.}, {29., 2.}, {31., 2.}, {44., 3.},
            {57., 4.}, {70., 5.}}
\end{verbatim}
By the previous results, this semigroup is Cohen-Macaulay, Gorenstein and Buchsbaum. These properties can be checked externally by using {\tt Macaulay2} with the following commands (see \cite{M2} for computing with {\tt Macaulay2}):
\begin{verbatim}
installPackage("MonomialAlgebras")
V={{3, 0}, {4, 0}, {5, 0}, {16, 1}, {17, 1}, {18, 1}, {29, 2},
     {31, 2}, {44, 3}, {57, 4}, {70, 5}}
isCohenMacaulayMA V
isGorensteinMA V
isBuchsbaumMA V
\end{verbatim}
All the outputs obtained are {\tt true}.
\end{example}


\begin{thebibliography}{20}
\bibitem{Assi}
A.~Assi.
\newblock The Frobenius vector of a free affine semigroup.
\newblock {\em Journal of Algebra and Its Applications}, 11(4): 1250065 (10 pages), 2012.

\bibitem{1311.1988}
A.~Assi, P.~A. Garc\'{\i}a-S\'{a}nchez, and I.~Ojeda.
\newblock Frobenius vectors, hilbert series and gluings.
\newblock ArXiv:1311.1988 [math.AC].

\bibitem{MR0153708}
H.~Bass.
\newblock On the ubiquity of {G}orenstein rings.
\newblock {\em Math. Z.}, 82:8--28, 1963.

\bibitem{MR1600261}
E.~Briales, A.~Campillo, C.~Mariju\'{a}n, and P.~Pis\'{o}n.
\newblock Minimal systems of generators for ideals of semigroups.
\newblock {\em J. Pure Appl. Algebra}, 124(1-3):7--30, 1998.



\bibitem{MR1251956}
W.~Bruns and J.~Herzog.
\newblock {\em Cohen-{M}acaulay rings}, volume~39 of {\em Cambridge Studies in
	Advanced Mathematics}.
\newblock Cambridge University Press, Cambridge, 1993.


\bibitem{MR3110598}
J.~I. Garc{\'{\i}}a-Garc{\'{\i}}a, M.~A. Moreno-Fr{\'{\i}}as,
  A.~S\'{a}nchez-R.-Navarro, and A.~Vigneron-Tenorio.
\newblock Affine convex body semigroups.
\newblock {\em Semigroup Forum}, 87(2):331--350, 2013.


\bibitem{MR3212609}
J.~I. Garc{\'{\i}}a-Garc{\'{\i}}a and A.~Vigneron-Tenorio.
\newblock Computing families of {C}ohen-{M}acaulay and {G}orenstein rings.
\newblock {\em Semigroup Forum}, 88(3):610--620, 2014.

\bibitem{ProporcionallyModularAffineSemigroupN2}
J.~I. Garc{\'{\i}}a-Garc{\'{\i}}a and A.~Vigneron-Tenorio.
\newblock ProporcionallyModularAffineSemigroupN2, a software system to solve a
proportionally modular inequality in {$\N^2$}.
\newblock Available at
\url{http://departamentos.uca.es/C101/pags-personales/alberto.vigneron/p_m_a_s_n2.zip}.


\bibitem{RosalesBuchs}
P.~A. Garc{\'{\i}}a-S\'{a}nchez and J.~C. Rosales.
\newblock On {B}uchsbaum simplicial affine semigroups.
\newblock {\em Pacific J. Math.}, 202(2):329--339, 2002.

\bibitem{MR602063}
S.~Goto.
\newblock On {B}uchsbaum rings.
\newblock {\em J. Algebra}, 67(2):272--279, 1980.

\bibitem{M2}
D.~R. Grayson and M.~E. Stillman.
\newblock Macaulay2, a software system for research in algebraic geometry.
\newblock Available at \url{http://www.math.uiuc.edu/Macaulay2/}.

\bibitem{MR1732040}
C.~Huneke.
\newblock Hyman {B}ass and ubiquity: {G}orenstein rings.
\newblock In {\em Algebra, {$K$}-theory, groups, and education ({N}ew {Y}ork,
  1997)}, volume 243 of {\em Contemp. Math.}, pages 55--78. Amer. Math. Soc.,
  Providence, RI, 1999.

\bibitem{Nsolutions}
P.~Pis\'{o}n-Casares and A.~Vigneron-Tenorio.
\newblock {$\Bbb N$}-solutions to linear systems over {$\Bbb Z$}.
\newblock {\em Linear Algebra Appl.}, 384:135--154, 2004.

\bibitem{MR3231518}
A.~M. Robles-P\'{e}rez and J.~C. Rosales.
\newblock Proportionally modular numerical semigroups with embedding dimension
  three.
\newblock {\em Publ. Math. Debrecen}, 84(3-4):319--332, 2014.

\bibitem{MR2753834}
J.~C. Rosales.
\newblock The maximum proportionally modular numerical semigroup with given
  multiplicity and ratio.
\newblock {\em Semigroup Forum}, 82(1):83--95, 2011.

\bibitem{MR2020273}
J.~C. Rosales, J.~I. Garc{\'{\i}}a-S\'{a}nchez, P. A.
  Garc{\'{\i}}a-Garc{\'{\i}}a, and J.~M. Urbano-Blanco.
\newblock Proportionally modular {D}iophantine inequalities.
\newblock {\em J. Number Theory}, 103(2):281--284, 2003.

\bibitem{RosalesCM}
J.~C. Rosales and P.~A. Garc{\'{\i}}a-S\'{a}nchez.
\newblock On {C}ohen-{M}acaulay and {G}orenstein simplicial affine semigroups.
\newblock {\em Proc. Edinburgh Math. Soc. (2)}, 41(3):517--537, 1998.

\bibitem{MR2389850}
J.~C. Rosales, P.~A. Garc{\'{\i}}a-S\'{a}nchez, and J.~M. Urbano-Blanco.
\newblock The set of solutions of a proportionally modular {D}iophantine
  inequality.
\newblock {\em J. Number Theory}, 128(3):453--467, 2008.

\bibitem{MR2796773}
J.~C. Rosales and J.~M. Urbano-Blanco.
\newblock Irreducible proportionally modular numerical semigroups.
\newblock {\em Publ. Math. Debrecen}, 78(2):359--375, 2011.

\bibitem{MR874114}
A.~Schrijver.
\newblock {\em Theory of linear and integer programming}.
\newblock Wiley-Interscience Series in Discrete Mathematics. John Wiley \&
  Sons, Ltd., Chichester, 1986.
\newblock A Wiley-Interscience Publication.

\bibitem{MR881220}
J.~St\"{u}ckrad and W.~Vogel.
\newblock {\em Buchsbaum rings and applications. An interaction between algebra, geometry and topology}.
\newblock Springer-Verlag, Berlin, 1986.

\end{thebibliography}
\end{document}